\def\WITHAPP{WITHAPP}
\def\NOAPP{NOAPP}

\def\form{WITHAPP}

\documentclass[12pt]{amsart}

\usepackage[utf8]{inputenc}
\usepackage{color}
\usepackage[centertags]{amsmath}
\usepackage[margin=1in]{geometry}
\usepackage{amsfonts}
\usepackage{amsthm}
\usepackage{caption}
\usepackage{newlfont}
\usepackage{verbatim}
\usepackage{amscd}
\usepackage{amsgen}
\usepackage{xspace}
\usepackage{amssymb}
\usepackage{stmaryrd}
\usepackage{amssymb,amsmath}
\usepackage{amscd}
\usepackage{enumitem}
\usepackage{color}
\usepackage{xy}
\usepackage{graphicx}
\usepackage{tikz}
\usetikzlibrary{arrows,shapes,trees}
\usetikzlibrary{backgrounds}
\usetikzlibrary{positioning}

\newlength{\defbaselineskip} 
\theoremstyle{plain}
\newtheorem{thm}{Theorem}[section]
\newtheorem{cor}[thm]{Corollary}
\newtheorem{con}[thm]{Conjecture}
\newtheorem{lema}[thm]{Lemma}
\newtheorem{prop}[thm]{Proposition}
\newtheorem{obs}[thm]{Observation}

\theoremstyle{definition}
\newtheorem{df}[thm]{Definition}
\newtheorem{exm}[thm]{Example}
\newtheorem{rem}[thm]{Remark}

\newtheorem{fact}[thm]{Fact}

\newtheorem{pr}{Algorithm}

\numberwithin{equation}{section}
\def\p{\mathbb{P}}

\def\RR{\mathbb{R}}

\def\Z{\mathbb{Z}}
\def\Q{\mathbb{Q}}
\def\N{\mathbb{N}}

\def\C{\mathbb{C}}

 \DeclareMathOperator{\Proj}{Proj}

\DeclareMathOperator{\Spec}{Spec}
\def\p{\mathbb{P}}

\def\ob{\begin{obs}}
\def\kob{\end{obs}}
\def\dow{\begin{proof}}
\def\kdow{\end{proof}}

\def\tw{\begin{thm}}
\def\ktw{\end{thm}}
\def\hip{\begin{con}}
\def\khip{\end{con}}
\def\lem{\begin{lema}}
\def\klem{\end{lema}}
\def\ex{\begin{exm}}
\def\prog{\begin{pr}}
\def\kprog{\end{pr}}
\def\wn{\begin{cor}}
\def\kwn{\end{cor}}
\def\uwa{\begin{rem}}
\def\kuwa{\end{rem}}
\def\kex{\end{exm}}
\def\dfi{\begin{df}}
\def\kdfi{\end{df}}
\def\W{\mathcal W}
\definecolor{zielony}{rgb}{0.5, 0.9, 0.1}
\definecolor{czerwony}{rgb}{0.9, 0.2, 0.1}
\definecolor{niebieski}{rgb}{0.3, 0.1, 0.9}

\newcommand{\barvinok}{\textsc{barvinok}\xspace}
\newcommand{\iscc}{\textsc{iscc}\xspace}
\newcommand{\isl}{\textsc{isl}\xspace}
\xyoption{line}
\def\fa{\begin{fact}}
\def\kfa{\end{fact}}
\DeclareMathOperator{\sgn}{sgn}
\DeclareMathOperator{\vol}{Vol}

\usepackage{url}
\usepackage{hyperref}
\definecolor{darkblue}{RGB}{0,0,160}
\hypersetup{
colorlinks,%
citecolor=black,%
filecolor=black,%
linkcolor=darkblue,%
urlcolor=darkblue
}

\title{Plethysm and lattice point counting}

\author{Thomas Kahle}
\address{Fakult\"at f\"ur Mathematik, Otto-von-Guericke Universit\"at\\39106 Magdeburg, Germany}
\urladdr{\url{http://www.thomas-kahle.de}}

\author{Mateusz Micha\l ek} \address{Polish Academy of Sciences,
\'Sniadeckich 8, 00-656 Warsaw, Poland, and Simons Institute for the
Theory of Computing, 121 Calvin Lab 2190, UC Berkeley, CA 94720, USA }
\email{wajcha2@poczta.onet.pl}

\thanks{Mateusz Micha\l ek is supported by Polish National Science Center grant
no. 2012/05/D/ST1/01063}
\makeatletter
  \@namedef{subjclassname@2010}{\textup{2010} Mathematics Subject Classification}
\makeatother

\date{August 2015}

\begin{document}
\begin{abstract}
We apply lattice point counting methods to compute the multiplicities
in the plethysm of $GL(n)$. Our approach gives insight into the
asymptotic growth of the plethysm and makes the problem amenable to
computer algebra.  We prove an old conjecture of Howe on the leading
term of plethysm.  For any partition $\mu$ of $3,4$, or~$5$ we obtain
an explicit formula in $\lambda$ and $k$ for the multiplicity of
$S^\lambda$ in~$S^\mu(S^k)$.
\end{abstract}

\maketitle

\setcounter{tocdepth}{1}
\tableofcontents

\section{Introduction}
The \emph{plethysm} problem can be stated in different ways.  One is
to describe the homogeneous polynomials on the spaces $S^k W^*$ and
$\bigwedge^k W^*$ in terms of representations of the group~$GL(W)$.
This is equivalent to decomposing $S^d(S^k W)$ into isotypic
components and finding the multiplicity of each isotypic component.
The general goal in plethysm is to determine the coefficients of
$S^\lambda$ in $S^\mu(S^\nu W)$ as a function of the partitions
$\lambda,\mu$, and~$\nu$.  The term plethysm was coined by
Littlewood~\cite{littlewood1936polynomial}, and this type of problems
appears in many branches of mathematics beyond representation theory
(consult~\cite{loehr2011computational} for some recent developments in
\emph{plethystic calculus}).  A general explicit solution of plethysm
may be intractable as the resulting formulas are simply too
complicated.  Here we show piecewise quasi-polynomial formulas that
describe the plethysm and then focus on two directions.  One is
explicit descriptions for small $\mu$ which we find with the help of
computer algebra.  The other direction is asymptotics of plethysm
where we confirm a conjecture of Howe~\cite[3.6(d)]{Howe87} on the
lead term (Theorem~\ref{thm:asympt}).

Our contributions are summarized in the following theorem.  The proof
of the formula is complete after Section~\ref{sec:reductions}, while
the asymptotics is dealt with in Section~\ref{sec:asymptotic}.
\begin{thm}\label{t:big} Let $\mu$ be a fixed partition, $k$ a
natural number, and let $\lambda$ be a partition of~$k|\mu|$.  The
multiplicity of the isotypic component of $S^\mu(S^k W)$ corresponding
to $\lambda$, as a function of $\lambda$ and $k$, is the following
piecewise quasi-polynomial:
\begin{equation*}
\frac{\dim \mu}{|\mu|!}\# P_{k,|\mu|}^\lambda+
(-1)^{{|\mu|-1}\choose{2}}\Bigg(\sum_{\alpha\vdash |\mu|, \alpha\neq (1,\dots,1)}\chi_\mu(\alpha)\frac{D_\alpha}{|\mu|!}\sum_{\pi\in
S_{|\mu|-1}}\sgn(\pi) Q_\alpha(k,\lambda_\pi)\Bigg),
\end{equation*}
where $P_{k,|\mu|}^\lambda$ is an explicit polytope, $\chi_\mu$ is the
character of the symmetric group $S_{|\mu|}$ corresponding to the
partition $\mu$, the $Q_\alpha$ are counting functions for the fibers
of projections of explicit polyhedral cones, and $\lambda_\pi$ is a
linear shift of~$\lambda$. Moreover,
$\frac{\dim \mu}{|\mu|!}\# P_{k,|\mu|}^\lambda$ is the leading term
and can be interpreted as coming from the Littlewood--Richardson rule.
When $\mu$ is any partition of 3, 4, or~5, the explicit piecewise
quasi-polynomials have been computed and can be downloaded from the
project homepage in \isl format.
\end{thm}

Theorem~\ref{t:big} yields explicit formulas for plethysm that
generalize known results in the cases $|\mu| = 2,3,4$.  Although these
formulas are not necessarily practical to work with on paper,
computers are quick to evaluate them, study their asymptotics, and
generally extract different sorts of information from them.  In this
sense, Theorem~\ref{t:big} is more effective (but maybe less
instructive) than approaches by tableaux counting such
as~\cite{rush2014cyclic}. Although its individual constituents are
quasi-polynomials whose chambers are cones, this cannot be guaranteed
for the whole expression solely from the formula in
Theorem~\ref{t:big}.  We discuss this in Remarks~\ref{r:chambers}
and~\ref{r:symplectic}.

To arrive at the theorem, we first compute the character of the
representation~$S^\mu(S^k W)$.  Using known formulas relating Schur
polynomials and complete symmetric polynomials, we relate the
multiplicities of isotypic components of the plethysm to coefficients
of monomials of a specific polynomial
(Propositions~\ref{prop:plethysmChar} and~\ref{prop:expansion_coeff},
Section~\ref{s:latticepoints}).  We then reduce the determination of
these coefficients to a purely combinatorial problem: lattice point
counting in certain rational polytopes related to transportation
polytopes (Definition~\ref{def:almatrix}).  For fixed $\mu$, the final
multiplicity is a function of $\lambda_1,\dots,\lambda_{|\mu|}$
and~$k$.  These arguments may belong to a finite number of polyhedral
chambers.  In each chamber, the result is a quasi-polynomial, that is,
a polynomial with coefficients that depend on the remainders of its
arguments modulo a fixed number.  Equivalently, it is a polynomial in
floor functions of linear expressions in the arguments.  Software to
determine piecewise quasi-polynomials is well-developed due to
applications ranging from toric geometry to loop optimization in
compiler research.  We show how to use
\barvinok~\cite{verdoolaege2007counting} and the
\isl-library~\cite{verdoolaege2010isl} to make Theorem~\ref{t:big}
explicit.  This yields a concrete decomposition of $S^\mu(S^k W)$ (and
$S^\mu(\bigwedge^k W)$) for any partition $\mu$ of 3, 4, or~5.  For
each fixed $\mu$, the result is a decomposition of
$(\lambda, k)$-space into polyhedral chambers, such that in each
chamber the multiplicity is a quasi-polynomial.  We have set up a
homepage for the results in this paper~at
\begin{center}
\url{http://www.thomas-kahle.de/plethysm.html}
\end{center}
\ifx\form\NOAPP The arXiv version of this paper has an appendix where
we detail our experiences with the software.\fi\ifx\form\WITHAPP In
the appendix (Section~\ref{sec:appendix}) we detail our experiences
with the software.\fi\xspace Our computations have been carried out
with version 0.37 of \barvinok and version 0.13 of~\isl.

Before presenting our methods, we now give a short overview of
applications of our results as well as different approaches.
\begin{description}[leftmargin=0cm, itemsep=1ex]
\item[Classical results] Our results extend classical theory.  For
example, the description of quadrics on the space $S^k(W^*)$ is a
classical result of Thrall~\cite{Thrall42}, \cite[4.1--4.6]{CGR83}.
\begin{exm}\label{ex:drugasym}
One has $GL(W)$-module decompositions
\[
S^2(S^k W)=\bigoplus S^\lambda W,\qquad
\textstyle\bigwedge^2\displaystyle(S^k W)=\bigoplus S^{\delta} W,
\]
where the first sum runs over representations corresponding partitions
$\lambda$ of $2k$ into two even parts and the second
sum runs over representations corresponding to partitions $\delta$ of
$2k$ into two odd parts.
\end{exm}
The decomposition of cubics $S^3(S^k W)$ is also known. Stated in
different forms, it can be found in \cite{Thrall42, MR0294526, CGR83,
Howe87, agaoka2002decomposition}. In fact, the latter four have
formulas for $S^\mu(S^k W)$ for any partition $\mu$ of~$3$.  The
determination of $S^4(S^k)$ has been addressed
in~\cite{MR0062098,MR0051212, Howe87}.

\item[Asymptotics] The explicit formulas for plethysm become
complicated quickly, but there is hope for simpler asymptotic
formulas.  For instance, the decomposition of $S^d(S^k W)$ is related
to $(S^k W)^{\otimes d}$ by means of the symmetrizing operator
$(S^k W)^{\otimes d}\rightarrow S^d(S^k W)$.  There the decomposition
of the domain of the resulting quasi-polynomial is known from Pieri's
(or more generally the Littlewood--Richardson) rule.  In the same
vein, Howe \cite[3.6(d)]{Howe87} identified the leading term for
$S^3(S^k)$ and $S^4(S^k)$.  A different approach by Fulger and
Zhou~\cite{fulger2012asymptotic} studies the asymptotics of plethysm
by considering how many different irreducible representations and
which sums of multiplicities can appear.  Further asymptotic results,
e.g.,~when the inner Schur functor is fixed, are presented
in~\cite{christandl2014eigenvalue}.  They are achieved through a
connection to the commutation of quantization and
reduction~\cite{sjamaar1995holomorphic, meinrenken1999singular,
meinrenken1996riemann}.

Knowledge of explicit quasi-polynomial formulas allows one to test
techniques for studying the asymptotics algorithmically on nontrivial
examples.  Another insight from Section~\ref{sec:asymptotic} is that
the language of convex discrete geometry may be more useful for proofs
than that of piecewise quasi-polynomials.

\item[Evaluation] One of the principal uses of our results is
evaluation of the plethysm function.  While evaluation for individual
values can be done in \verb|LiE|~\cite{van1992lie} and other packages,
our results are more flexible as they are given as functions on
parameter space and can thus be evaluated parametrically.

\begin{exm}\label{ex:largeeval}
Let $\mu=(5)$, $\lambda=(31,3,2,2,2)$, and make the following
definitions:
\begin{gather*}
p_1 = - \frac{289}{720} s + \frac{1}{20} s^2 + \frac{1}{720} s^3\\
p_2 = \frac{5}{8} + \frac{1}{8} s, \qquad
p_3 = \frac{1}{3} - \frac{1}{6} s, \qquad
p_4 = \frac{7}{12} - \frac{1}{3} s,  \\
A(s) = p_1 + p_2 \left\lfloor \frac{s}{2}\right\rfloor + p_3
\left\lfloor\frac{s}{3}\right\rfloor + \left( p_4 + \frac{1}{2}
  \left\lfloor\frac{s}{3}\right\rfloor \right)
\left\lfloor\frac{1+s}{3}\right\rfloor + \frac{1}{4} \left(
  \left\lfloor\frac{1 + s}{3}\right\rfloor^2 +
  \left\lfloor\frac{s}{4}\right\rfloor - \left\lfloor\frac{3 +
    s}{4}\right\rfloor \right)
\end{gather*}
With these definitions the coefficient of $S^{s\lambda}$ in
$S^\mu(S^{8s})$ equals
\begin{equation*}
A(s) +
\begin{cases}
1 & \text{ if } s \equiv 0 \mod 5\\
\frac{3}{5} & \text{ if } s \equiv 1 \mod 5\\
\frac{4}{5} & \text{ if } s \equiv 2,3,4 \mod 5,
\end{cases}
\end{equation*}
Note that after the reductions in Section~\ref{sec:reductions},
Lemma~\ref{l:qpolyasymptotics} with $l=2, d=5$ gives a degree bound of
three, which is realized here.  Also note that the result is a single
quasi-polynomial (see Remark~\ref{r:chambers}).
\end{exm}

\begin{rem}\label{rem:tricky}
It can be slightly tricky to automatically extract these kinds of
formulas from our computational results.  In general, \iscc has
somewhat limited capabilities in producing human-readable output.
\ifx\form\WITHAPP In Example~\ref{ex:case-study} \fi \ifx\form\NOAPP
In the appendix of the arXiv version of this paper, \fi we give a
complete discussion of how to derive this result using our
computational results and~\iscc.
\end{rem}

\begin{exm}
Actual numerical evaluation is very quick too.  For instance, the
multiplicity of the isotypic component of
$$\lambda=(616 036 908 677 580 244,1 234 567 812 345 678,12 345 671
234 567,123 456 123 456)$$
in $S^5(S^{123 456 789 123 456 789})$ equals \small
\begin{equation*}
24 096 357 040 623 527 797 673 915 801 061 590 529 381 724 384 546 352 415 930 440 743 659 968 070 016 051.
\end{equation*} \normalsize
The evaluation of our formula on this example takes under one second,
and this time is almost entirely constant overhead for dealing with
the data structure.  Evaluation on much larger arguments (for instance
with a million digits) is almost as quick.
\end{exm}

\item[Quantum physics] Descriptions of entangled quantum states of
bosons and fermions are related to plethysms of symmetric and wedge
product (see~\cite{christandl2014eigenvalue,
christandl2012computing}).

\item[Testing conjectures] Although many theoretical formulas for
plethysm are known, some basic properties are mysterious.  For
example, a conjecture of Foulkes states that for $a < b$,
$S^{a}(S^{b})$ embeds as a subrepresentation into $S^{b}(S^{a})$.  For
$a=2$, this is a classical result. For $a=3$, it was shown in this
century~\cite{Dent2000236} and for $a=4$ in~\cite{TomMcKay2007}.  A
variant has been studied in~\cite{abdesselam2007brill}.  Using
explicit quasi-polynomials, one can attack this problem for any
fixed~$a$.  One would need to compare two explicit quasi-polynomials
and in particular decide whether their difference is a positive
quasi-polynomial.  At the moment, we cannot complete this direction as
our methods only work for fixed exponent of the outer Schur functor.
Results of Bedratyuk \cite{bedratyuk} indicate that explicit
quasi-polynomials can also be found for fixed exponent of the inner
Schur functor, once we fix the group (to be $SL(n)$). Our result can
also be considered as a step toward Stanley's Problem 9 in
\cite{Stanley} asking for the combinatorial description of plethysm.
For this major breakthrough one would need ``positive'' formulas,
though.

\item[The zero locus of plethysm coefficients] The question of which
isotypic components appear in plethysm is highly nontrivial. Some very
special cases follow from the resolution of Weintraub's conjecture
\cite{MR1037395,MR2745569,manivel2012effective}.  We hope that our
formulas can contribute to finding further regularities among
partitions that appear in different plethysms, for instance, by
studying the zeros of our quasi-polynomials.

\item[Geometric Complexity Theory] The problem of separation of
complexity classes is addressed with geometric methods
in~\cite{mulmuley2001geometric, burgisser2011overview}.  The crucial
point of this program requires comparing closures of orbits of
explicit symmetric tensors. The plethysm plays an important role
there~\cite[p.~516]{mulmuley2001geometric},
\cite[p.~10]{burgisser2011overview},
\cite[p.~10]{landsberg2014geometric}.

\item[Computation of syzygies] The computations of syzygies of
homogeneous varieties is related to (inner) plethysm (see
\cite[p.~63]{Weyman}).  Weyman \cite[p.~241]{Weyman} applies the
explicit computation of plethysm $S^n(S^2)$ to study rank varieties
\cite[Section~7.1]{Weyman} and topics related to free resolution of
the Grassmannian.

\item[Unification] Many specialized methods have been developed to
attack plethysm problems, and most of them work only in a very
restricted set of exponents.  For instance, we have already mentioned
several methods to compute $S^3(S^k)$ \cite{Thrall42, MR0294526,
CGR83, Howe87, agaoka2002decomposition}.  Quasi-polynomials and convex
bodies provide a unifying framework for all of those
techniques. Specializing to certain situations is just partially
evaluating the quasi-polynomial.  In the history of plethysm, people
have written a new paper with a new technique whenever the next
exponent was due.  Our method, in contrast, stays the same.

\item[Representations of $S_n$] The plethysm is also related to
representations of the permutation group~$S_n$.  For two
representations of $S_a$ and $S_b$ corresponding to $\lambda$ and
$\mu$, respectively, there is a wreath product representation
$\lambda \wr \mu$ of the wreath product group $S_a\wr S_b$.  One has a
natural inclusion $S_a\wr S_b\subset S_{ab}$.  By the Frobenius
characteristic map, we can identify representations of symmetric
groups with symmetric polynomials, and under this identification, the
representation of $S_{ab}$ induced from $\lambda\wr\mu$ is exactly the
plethysm of the representations given by $\lambda$ and $\mu$.  The
interested reader may consult \cite[Vol. 2, Theorem A2.6]{Stanley} and
references there.  For similar results on the Kronecker coefficients,
appearing in the decomposition of the tensor product of
representations of symmetric groups,
see~\cite{briand2009quasipolynomial, briand2011stability}.

\item[Classical algebraic geometry] The spaces $S^k W^*$ and
$\bigwedge^k W^*$ are ambient spaces of the Veronese variety and the
Grassmannian. These varieties and related objects, e.g.,~their secant
and tangential varieties, have been studied classically (see
\cite{MR1234494} and references therein). The description of the
algebra of the Veronese and Grassmannian is well-known. However, as
the decomposition of $S^d(S^k W)$ is not known, the decomposition of
the degree $d$ part of the ideal is a difficult problem---even for
quartics! Our results provide such a description. Furthermore, due to
problems motivated by determining ranks of tensors, secant varieties
are often studied from a computational point of view.  It is an open
problem to check, whether the ideal of the secant (line) variety of
any Grassmannian is generated by cubics. A~description of all cubics
in the ideal was given in~\cite{manivel2014secants}. It is natural to
ask which quadrics are generated by cubics. To answer this question,
the description of all degree four polynomials is helpful. Thus, our
results provide very practical information. One could argue that we
provide the decomposition of very low degree equations. Note, however,
that on the $k$-th secant variety, no equations of degree less than or
equal to $k$ vanish. On the other hand, the equations of degree $k+1$
sometimes already provide all generators of the ideal (e.g., the
Segre-Veronese varieties for $k=2$ \cite{Raicu}). Thus, knowing their
decompositions is an important first step in determining the structure
of the whole ideal.
The same method can be applied to other ideals defined by objects
related to representation theory. One example is the ideal of
relations among $k\times k$ minors of a generic matrix studied
in~\cite{BCV}.

\item[Errors] As the formulas and computations become more and more
technically involved, the chance of human error rises.  To quote from
Howe~\cite{Howe87}: \emph{Here we will outline what is involved in the
computations and list our answers. The details are available from the
author on request. The author does hope someone will check the
calculations, because he does not have a great deal of faith in his
ability to carry through the details in a fault-free manner. He hopes,
however, that the answers are qualitatively correct as stated.}  We
have not checked all historical formulas that overlap with our
results, but errors have been identified before
(compare~\cite[Appendix]{manivel2014secants} and~\cite{CGR83}).
\end{description}

\subsection*{Convention and notation}
All representations considered are finite dimensional. Let
$\lambda=(\lambda_1,\dots,\lambda_l)$ with
$\lambda_1\geq\dots\geq \lambda_l> 0$ be an integral partition of $n$,
i.e.,~$\sum_{i=1}^l\lambda_i=n$. We set $|\lambda|=n$. Consider a
vector space $W$ of dimension at least~$l$. Let $S^\lambda W$ be the
irreducible representation of $GL(W)$ corresponding to~$\lambda$,
obtained by acting with the Young symmetrizer $c_\lambda$ on
$W^{\otimes n}$~\cite{fulton91:_repres_theor}. We use the convention
that the partition $(1,\dots,1)$ corresponds to the wedge product
representation $\bigwedge^{n}W$ and $(\lambda_1)$ to the symmetric
power $S^n W$. All irreducible representations of $SL(W)$ can be
obtained by considering partitions $\lambda$ with $n$ arbitrary and
$l<\dim W$.  For $GL(W)$ the theory is similar.  There we have to
specify an additional integer~$r$. The vector space and the group
action are the same as those for $SL(W)$, but additionally we multiply
a given vector by the determinant to the power~$r$.

In general, the \emph{Schur functors} $S^\lambda$ are endofunctors of
the category of representations.  Applying them to the standard
representation $W$ yields all irreducible representations.  Applying
$S^\lambda$ to other irreducible representations, in general, yields
reducible representations.  The \emph{plethysm} is to understand the
decomposition of~$S^\lambda(S^\mu W)$.

\subsection*{Acknowledgments}
The authors would like to thank Sven Verdoolaege for his prompt
responses to issues raised on the \isl-mailing list. The second author
would like to thank Laurent Manivel for introducing him to the subject
of plethysm.  After the first posting of this paper on the arXiv
Matthias Christandl, Laurent Manivel, and Mich\`ele Vergne provided
very insightful comments on how to apply Meinrenken-Sjamaar theory to
the plethysm.  We would like to thank them also for their suggestions
on how to improve the paper.  This project started while Micha\l{}ek
was an Oberwolfach Leibniz fellow and invited Kahle for work at
MFO. The project was finished at Freie Universit\"at Berlin during
Micha\l{}ek's DAAD PRIME fellowship.

\section{Characters}
The trace of a $GL(W)$ representation is a symmetric polynomial in the
eigenvalues, known as the \emph{character} of the representation. To
determine the character, one considers the action of diagonal
matrices.

\begin{df}[$h_k(x^a)$, $\psi_\alpha$]
Consider $d$ variables $x_1,\dots,x_d$.  For $a\in\N$, let $h_k(x^a)$
be the complete symmetric polynomial of degree $k$ in the variables
$x_1^a,\dots,x_d^a$.  Let $\alpha$ be a multi-index of length~$j$.  We
define the polynomials
\[
\psi_n=\sum_i x_i^n,\qquad \psi_\alpha=\prod_{i=1}^j \psi_{\alpha_i}, \quad
\text{and} \quad
\psi_{\alpha}\circ h_k = \prod_{i=1}^j h_k(x^{\alpha_{i}}).
\]
\end{df}
\begin{exm}
The character of the representation $S^k(W)$ is the sum of all
monomials of degree $k$ in $\dim W$ variables, that is~$h_k(x)$.  For
$\bigwedge^k W$, we obtain the sum of all square-free monomials of
degree~$k$, known as the elementary symmetric polynomial.
\end{exm}

For any representation $V$, the associated character is denoted
by~$P_V$.  The character of the irreducible representation
$S^\lambda W$ is the \emph{Schur polynomial}~$P_\lambda$.  Schur
polynomials are independent and form a basis of symmetric polynomials.
Since the character of a sum of two representations is the sum of
their characters, in order to decompose any representation, it is
enough to express its character as a sum of Schur polynomials.
Precisely, $V=\sum (S^\lambda V)^{\oplus a_\lambda}$ if and only if
$P_V=\sum a_\lambda P_\lambda$.  Other operations on representations
translate too.  For instance, the plethysm of two symmetric
polynomials $f,g$ is the composition $f\circ g$
(see~\cite[I.8]{macdonald1998symmetric} for a precise algebraic
definition).

\begin{prop}[{\cite[I.8.3, I.8.4, I.8.6]{macdonald1998symmetric}}]
\label{prop:pletyzm}
For any symmetric polynomial $f$, the mapping $g\rightarrow g\circ f$
is an endomorphism of the ring of symmetric polynomials. For any
$n\in\N$, the mapping $g\rightarrow \psi_n\circ g$ is an endomorphism
of the ring of symmetric polynomials. Moreover,
  \[
  \psi_n\circ g=g\circ \psi_n=g(x_1^n,x_2^n,\dots).
  \]
\end{prop}
\begin{rem}
\label{Rem:Notation} Proposition~\ref{prop:pletyzm} justifies the
notation
\[
\psi_{\alpha}\circ h_k=\prod_{i=1}^j h_k(x^{\alpha_{i}}).
\]
\end{rem}

From now on, assume that $\dim W$ is large enough so that all
appearing partitions have at most $\dim W$ parts and fix a partition
$\mu$ of an integer~$d$.  Irreducible representations of the
permutation group $S_d$ are indexed by Young diagrams with exactly $d$
boxes. The character corresponding to the Young diagram $\rho\vdash d$
is denoted~$\chi_\rho$.
\begin{df}[{$z_\rho$, \cite[p.17]{macdonald1998symmetric}}] Let
$\rho=(\rho_1\geq \dots\geq\rho_k)$ be a partition of~$d$ and $m_i$
the number of parts equal to~$i$. We define
\[
z_\rho = \prod_{i\geq 1}i^{m_i}m_i! = \frac{d!}{D_\rho},
\]
where $D_\rho$ is the number of permutations of cycle type $\rho$.
\end{df}

\begin{rem}
\label{uwa:formula}
With \cite[I.7.(7.2) and I.7.(7.5)]{macdonald1998symmetric} the
character $P_\mu$ can be expressed in terms of $\psi_n$ as
\[
P_\mu=\sum_{\rho\vdash d}z_\rho^{-1}\chi_\mu(\rho)\psi_\rho,
\]
where $\chi_\mu(\rho)$ is the value of the character $\chi_\mu$ on
(any) permutation of type $\rho$.
\end{rem}
\begin{exm}
As the Young diagram $(d)$ corresponds to the trivial representation
of $S_d$, we obtain the formula for the complete symmetric polynomial:
\[
h_d=P_{(d)}=\sum_{\rho\vdash d}\frac{D_\rho}{d!}\psi_\rho,
\]
where $D_\rho$ is the number of permutations of combinatorial type
$\rho$ in the group~$S_d$.  The Young diagram $(1,\dots,1)$
corresponds to the sign representation of $S_d$. Hence, we obtain the
formula for the character of the wedge power:
\[
P_{(1,\dots,1)}=\sum_{\rho\vdash
d}\sgn(\rho)\frac{D_\rho}{d!}\psi_\rho.
\]
\end{exm}
\begin{prop}\label{prop:plethysmChar}
The character of the representation $S^\mu(S^k W)$ equals
\[
P_{S^\mu(S^k W)}=\sum_\alpha
\chi_\mu(\alpha)\frac{D_\alpha}{d!}\psi_\alpha\circ h_k,
\]
where the sum is taken over all partitions $\alpha$ of $d:=|\mu|$ and
$D_\alpha$ is the number of permutations of cycle type $\alpha$ in the
group~$S_d$.
\end{prop}
\begin{proof}
We have
\[
P_{S^\mu(S^k W)}=P_{S^\mu}\circ h_k.
\]
By Remark~\ref{uwa:formula}, this equals
\[
\sum_{\rho\vdash d}z_\rho^{-1}\chi_\mu(\rho)\psi_\rho\circ h_k.  \qedhere
\]
\end{proof}

\begin{rem}\label{r:yangKostka}
A similar formula for arbitrary composition of Schur functors is
presented in \cite[Theorem 2.2]{yang1998algorithm}. We do not apply it
directly, as it relies on 'nested inverse Kostka numbers'. As
explained in \cite{yang1998algorithm, yang2002first}, the computation
of those, although possible in many cases, is a nontrivial task. For
this reason, we introduce one more change of basis of symmetric
polynomials, relating our results to transportation polytopes. From
the algorithmic point of view, although the final result counts the
same multiplicities, enumeration of points in dilated polytopes is
easier than enumeration of skew Young diagrams with specific
properties.
\end{rem}

For fixed $d$, all partitions can be listed and the decomposition of
$P_{S^\mu(S^k W)}$ into Schur polynomials reduces to the decomposition
of each polynomial $\psi_\alpha\circ h_k$. Indeed, the values of
$\chi_\mu(\rho)$ can be made explicit by the celebrated Frobenius
Formula \cite[4.10]{fulton91:_repres_theor}.  As similar results will
be used later, we review the formula in detail.
\begin{df}[{$[P]_\alpha$}, $\Delta(x)$]
For any polynomial $P$ and partition
$\alpha=(\alpha_1,\dots,\alpha_k)$, define $[P]_\alpha$ as the
coefficient of the monomial $x_1^{\alpha_1}\cdot\dots\cdot
x_k^{\alpha_k}$ in~$P$.
\end{df}
\begin{df}
For a fixed number of variables $x_1,\dots,x_k$, the
\emph{discriminant} is
\[
\Delta(x) = \prod_{i<j}(x_i-x_j).
\]
\end{df}

The value of the character $\chi_\mu$ on any permutation of cycle type
$\rho$ equals:
\begin{equation}\label{eq:FrobeniusFormula}
\chi_\mu(\rho)=[\Delta(x)\psi_\rho]_{(\mu_1+k-1,\mu_2+k-2,\dots,\mu_k)}.
\qquad(\text{Frobenius formula})\qquad
\end{equation}
\begin{exm}
Consider a permutation $\pi\in S_4$ of cycle type $(3,1)$, e.g.,~the
permutation that fixes $4$ and permutes
$1\rightarrow 2\rightarrow 3\rightarrow 1$. Consider a representation
corresponding to the partition $2+2=4$. We obtain:
\begin{equation*}
\chi_{(2,2)}(\pi)=[(x_1-x_2)(x_1^3+x_2^3)(x_1+x_2)]_{(3,2)}=[x_1^5-x_1^3x_2^2+x_1^2x_2^3-x_2^5]_{(3,2)}=-1.
\end{equation*}
\end{exm}

\section{Reductions}\label{sec:reductions}

To make Proposition~\ref{prop:plethysmChar} effective, we employ the
following simplifications.
\begin{enumerate}
\item Reduction of the number of variables.
\item Application of the Littlewood--Richardson rule to the most
complicated term.
\item Change of basis of symmetric functions.
\item Reduction to combinatorics of polytopes.
\end{enumerate}
\subsection{Reduction of the number of variables}
Our aim is to compute the multiplicity of the isotypic component
corresponding to $\lambda$ inside $S^\mu(S^k W)$. By the
Littlewood--Richardson rule, $\lambda$ can have at most $|\mu|$ rows,
so we can assume $\dim W=|\mu|$.
\begin{prop} [\cite{Carre}, \cite{MR1651092}]\label{prop:dualities}
\[
S^{\mu}(S^{2l} W)=S^{\mu}(\bigwedge^{2l} W)^\vee,\qquad
S^{\mu}(S^{2l+1} W)=S^{\mu^\vee}(\bigwedge^{2l+1}W)^\vee,
\]
where $(.)^\vee$ stands for the representation arising from $(.)$ by
replacing each irreducible component corresponding to a Young diagram
$\nu$ with the component corresponding to the transpose of $\nu$,
denoted~$\nu^\vee$.
\end{prop}
Proposition~\ref{prop:dualities} says that the multiplicity of an
isotypic component corresponding to $\lambda$ inside $S^{\mu}(S^{l}W)$
is the multiplicity of $\lambda^\vee$ inside either
$S^{\mu^\vee}(\bigwedge^{l}W)$ or $S^{\mu}(\bigwedge^{l} W)$.  For the
wedge power, the following well-known reductions hold which we prove
for the sake of completeness.
\begin{lema}[Reduction Lemma {\cite[5.8, 5.9]{Carre}},
{\cite[Lemma 6.3]{manivel2014secants}}]
\label{lem:redukcja} Let $\mu$ be any Young diagram of weight~$d$, and
$\lambda$ a Young diagram with $d$ columns and weight~$dk$. Let
$\lambda'$ equal $\lambda$ with the first row removed. The
multiplicity of the component corresponding to $\lambda$ in
$S^{\mu}(\bigwedge^k W)$ equals the multiplicity of the component
corresponding to $\lambda'$ in $S^{\mu}(\bigwedge^{k-1} W)$.
\end{lema}
\begin{proof}
Consider the inclusion $S^{\mu}(\bigwedge^k W)\subset (\bigwedge^k
W)^{\otimes d}$ with a basis given by tensor products of wedge
products of basis elements of~$W$. Each vector in the highest weight
space corresponding to $\lambda$ must contain exactly one $e_1$ in
each tensor. We get an isomorphism of highest weight spaces by
removing $e_1$ and decreasing the indices of other basis vectors by
one.
\end{proof}
The above facts show that whenever $\lambda^\vee$ has $d$ nonzero
columns (or equivalently $\lambda$ has $d$ nonzero rows), we can
express the multiplicity in the plethysm by a multiplicity in a
simpler plethysm.  It follows that it is enough to determine the
multiplicities of isotypic components corresponding to $\lambda$ with
at most $d-1$ rows.  This is equivalent to the assumption that
$\dim W=d-1$ or that the symmetric polynomials are in variables
$x_1,\dots,x_{d-1}$. From now on, we make this assumption, recovering
the general case at the end (Remark~\ref{r:extend}).

\subsection{Application of Littlewood--Richardson rule}
\label{sec:Littlewood-Richardson}
Suppose
\[
\psi_\alpha\circ h_k=\sum_\lambda a_{\alpha,\lambda} S^\lambda,
\]
where $S^\lambda$ is the Schur polynomial corresponding to $\lambda$
and the sum is over all partitions $\lambda \vdash dk$, with at most
$d-1$ parts.  In the following sections, we associate polytopes to the
polynomials $\psi_\alpha\circ h_k$.  Although our computer algebraic
methods work in general, they are least efficient for the partition
$\alpha=(1,\dots,1)$.  In this section, we show how to express
$\psi_{(1,\dots,1)}\circ h_k$ in terms of Schur polynomials without
further computation.  While in the end these reductions were not
necessary in our computations, we present them as an introduction to
the methods in the remaining sections and to better understand the
leading term in the plethysm formula.

Fix $\alpha_0=(1,\dots,1)\vdash d$. By Remark~\ref{Rem:Notation},
$\psi_{\alpha_0}\circ h_k=(h_k(x))^d,$ the $d$-th power of the
complete symmetric polynomial of degree $k$. As multiplication of
polynomials corresponds to the tensor product of representations, this
is the character of the representation $(S^k W)^{\otimes d}$.  The
decomposition of this representation is known due Pieri's rule (or
more generally the Littlewood--Richardson rule). In order to make the
formulas explicit, consider the following polytope.
\begin{df}[The polytope $P_{k,d}$]\label{def:pieriPoly}
Let $(x_1^1,x_1^2,x_2^2,\dots,x_1^{d-1},\dots,x_{d-1}^{d-1})$ denote
coordinates of the vector space
$\RR^1\times\RR^2\times\dots\times\RR^{d-1}$.  Denote $x^0_1=k$,
$x^j_{j+1}=k-\sum_{i=1}^j x_i^j$ and $x^j_i=0$ for $i>j+1$.  Let
$P_{k,d}$ be the polytope defined by the following constraints:
\begin{enumerate}
\item $x_i^j\geq 0$, for all $i,j$,
\item $\sum_{l\leq j}x_i^l\leq \sum_{l\leq j-1}x_{i-1}^l$, for all $j$
and $1<i\leq j+1$.
\end{enumerate}
\end{df}
In Definition~\ref{def:pieriPoly}, $x_i^j$ corresponds to the number
of boxes added according to Pieri's rule in the $j$-th step in the
$i$-th row.  For a polytope $P$, let $\# P$ denote the number of
integral points in~$P$.  By Pieri's rule we obtain the following
\begin{prop}\label{prop:expansion_coeff}
The coefficient $a_{\alpha_0,\lambda}$ in the expansion
\[
\psi_{\alpha_0}\circ h_k=\sum_\lambda a_{\alpha_0,\lambda} S^\lambda,
\]
equals the number of integral points in $P_{k,d}$ intersected with the
hyperplanes $\sum_j x_i^j=\lambda_i$.  In particular, it can be
computed as the number of points in the fiber of a projection of
$P_{k,d}$.  We will denote the intersection by $P_{k,d}^\lambda$. \qed
\end{prop}
\begin{rem}
There are other methods to compute the Littlewood--Richardson
coefficients, e.g.,~due to Berenstein and
Zelevinsky~\cite{berenstein1992triple}, that could provide other
polytopal descriptions. Contrary to plethysm, the question which
representations $S^\nu$ appear (with positive multiplicities) in
$S^\lambda\otimes S^\mu$ is well-understood \cite{klyachko1998stable,
knutson1999honeycomb, knutson2004honeycomb}.
\end{rem}
\subsection{Change of basis}
\label{sec:ChangeBasis}
Suppose
\[
\psi_\alpha\circ h_k=\sum_\lambda a_{\alpha,\lambda} S^\lambda,
\]
where $S^\lambda$ is the Schur polynomial corresponding to $\lambda$
and the sum is taken over all partitions $\lambda \vdash dk$, with at
most $d-1$ parts.  By the results of
\cite[Appendix~A]{fulton91:_repres_theor} and
\cite{macdonald1998symmetric}, the coefficient $a_{\alpha,\lambda}$ is
equal to the coefficient of the monomial $x_1^{\lambda_1+d-2}\cdots
x_{d-1}^{\lambda_{d-1}}$ in the polynomial $(\psi_\alpha\circ
h_k)\prod_{i<j} (x_i-x_j)$, that is:
\[
a_{\alpha,\lambda}=[\Delta(x)(\psi_\alpha\circ h_k)]_{(\lambda_1+d-2,\lambda_2+d-3,\dots,\lambda_{d-1})}
\]
\subsection{Integral points in polytopes}\label{s:latticepoints}
When $d$ is fixed, the discriminant $\prod_{i<j} (x_i-x_j)$ is
explicit.  Our aim is to compute the coefficients of the monomials
appearing in $\Delta(x)(\psi_\alpha\circ h_k)$.
\begin{df}[$(\alpha,\lambda)$-matrix]\label{def:almatrix}
Fix partitions $\alpha,\lambda$ and suppose that $\alpha$ has $a$
parts.  An $a\times (d-1)$ matrix $M$ with nonnegative integral
entries is an \emph{$(\alpha,\lambda)$-matrix} if
\begin{enumerate}
\item each row sums up to $k$, i.e.,~$\sum_{j=1}^{d-1} M_{i,j}=k$ for
each $1\leq i\leq a$, and
\item the $\alpha$-weighted entries of the $j$-th column sum up to
$\lambda_j$, i.e.,~$\sum_{i=1}^{a}\alpha_iM_{i,j}=\lambda_j$ for each
$1\leq j\leq d-1$.
\end{enumerate}
\end{df}
\begin{exm}
Let $d=3$ and $\alpha=(3)$ and
$\lambda = (\lambda_1,\lambda_2,\lambda_3)$.  According to
Definition~\ref{def:almatrix}, an $(\alpha,\lambda)$-matrix is a
nonnegative integral $(1\times 2)$ matrix $M = (M_{11}, M_{12})$
satisfying $M_{11} + M_{12} = k$ and $3M = (\lambda_1,\lambda_2)$.
There is no such matrix unless $\lambda_1 \equiv 0 \mod 3$, and if
this is the case, for each $k$, there is exactly one such matrix if
and only if $\lambda_2 = 3k - \lambda_1$.
\end{exm}
It is a straightforward observation that the coefficient of
$x^\lambda$ in $\psi_\alpha\circ h_k$ equals the number of different
$(\alpha,\lambda)$-matrices, as each matrix encodes the expansion of
the product $\prod_{i=1}^a h_k(x^{\alpha_i})$.  We want to obtain an
explicit formula for the number of $(\alpha,\lambda)$-matrices for
fixed $\alpha$ as a piecewise quasi-polynomial in
$k,\lambda_1,\dots,\lambda_{d-2}$ ($\lambda_{d-1}$ is determined as
$\sum_{i=1}^{d-1}\lambda_i=kd$).  Denote this quasi-polynomial by
$Q_\alpha$ such that
\[
\psi_\alpha\circ h_k=\sum_\lambda
Q_\alpha(k,\lambda_1,\dots,\lambda_{d-2})x^\lambda.
\]
Hence, by the Vandermonde formula,
\begin{align*}
\psi_\alpha\circ h_k\prod_{i<j} (x_i-x_j) & =\psi_\alpha\circ h_k
(-1)^{{d-1}\choose{2}}\prod_{i<j} (x_j-x_i) \\
& =(-1)^{{d-1}\choose{2}}(\sum_\lambda
Q_\alpha(k,\lambda_1,\dots,\lambda_{d-2})x^\lambda)(\sum_{\pi\in
S_{d-1}}\sgn(\pi) \prod_{i=1}^{d-1}x_i^{\pi(i)-1}).
\end{align*}
Consequently the coefficient of $x_1^{\lambda_1+d-2}\cdots
x_{d-1}^{\lambda_{d-1}}$ in $(\psi_\alpha\circ h_k)\prod_{i<j}
(x_i-x_j)$ equals:
\[
(-1)^{{d-1}\choose{2}}\sum_{\pi\in S_{d-1}}\sgn(\pi)
Q_\alpha(k,\lambda_1+d-1-\pi(1),\lambda_2+d-2-\pi(2),\dots,\lambda_{d-2}+2-\pi(d-2)).
\]
For each permutation $\pi \in S_{d-1}$, denote
$\lambda_\pi =
(\lambda_1+d-1-\pi(1),\lambda_2+d-2-\pi(2),\dots,\lambda_{d-2}+2-\pi(d-2))$.
Using this notation we obtain the formula for the multiplicity
$a_\lambda$ of the isotypic component corresponding to $\lambda$
inside $S^\mu(S^k W)$ for $\mu$ a partition of~$d$:
\[
(-1)^{{d-1}\choose{2}}\bigg(\sum_{\alpha\vdash d}\chi_\mu(\alpha)\frac{D_\alpha}{d!}\sum_{\pi\in
  S_{d-1}}\sgn(\pi)
Q_\alpha(k,\lambda_\pi)\bigg).
\]
The summand for the partition $\alpha=(1,\dots,1)$ can be made
explicit:
\begin{equation}\label{eq:formula}
\frac{\dim \mu}{d!}\# P_{k,|\mu|}^\lambda+
(-1)^{{d-1}\choose{2}}\bigg(\sum_{\alpha\vdash d, \alpha\neq (1,\dots,1)}\chi_\mu(\alpha)\frac{D_\alpha}{d!}\sum_{\pi\in
S_{d-1}}\sgn(\pi) Q_\alpha(k,\lambda_\pi)\bigg),
\end{equation}
where $\dim \mu=\chi_\mu(1,\dots,1)$ is the value of the character
$\chi_\mu$ on the trivial permutation and thus equal to the dimension
of the representation of $S_{|\mu|}$ corresponding to~$\mu$. We may
identify $S_{d-1}$ with the Weyl group~$\W$. Let $\rho$ be half of the
sum of positive weights. For esthetic reasons, we may rewrite the
above formulas as follows
\[
(-1)^{{d-1}\choose{2}}\left(\sum_{\alpha\vdash
  d}\chi_\mu(\alpha)\frac{D_\alpha}{d!}\sum_{\pi\in W}\sgn(\pi)
  Q_\alpha(k,\lambda+\rho-\pi(\rho))\right).
\]
All together, we have reduced the problem of finding the coefficients
of the plethysm to computing the piecewise quasi-polynomials
$Q_\alpha$ that count the number of $(\alpha,\lambda)$-matrices.  Let
$\alpha$ be a partition with $a$ parts.  The integral
$(a\times(d-1))$-matrices form an $a(d-1)$-dimensional lattice and the
linear equations in Definition~\ref{def:almatrix} define hyperplanes
in this lattice.  When $L$ denotes the resulting affine sublattice,
the $(\alpha,\lambda)$-matrices are simply the nonnegative integer
points in~$L$.  Alternatively, let $P_{\alpha,\lambda}$ (not to be
confused with $P_{k,d}$) be the (rational) polytope
$(L\otimes_\Z \Q) \cap \Q^{a(d-1)}_{\ge 0}$.  It is a polytope since
each coordinate is nonnegative and bounded from above
by~$\max~\lambda_i$.  The number of $(\alpha,\lambda)$-matrices equals
$\# P_{\alpha,\lambda}$, the number of integral points
in~$P_{\alpha,\lambda}$. It is also worth noting that for any
partition $\alpha$, the polytope $P_{\alpha,\lambda}$ can be obtained
from the $P_{(1,\dots,1),\lambda}$ by a series of hyperplane cuts
given by equalities of coordinates. The polytopes
$P_{(1,\dots,1),\lambda}$ are \emph{transportation polytopes}, well
studied objects in combinatorics and
optimization~\cite{klee1968facets, bolker1972transportation,
balinski1993signature, de2013combinatorics, liu2013perturbation}.

The following remark follows by combining
Proposition~\ref{prop:dualities} and Lemma~\ref{lem:redukcja}.
\begin{rem}\label{r:extend}
Let $\mu$ be a partition of $d$ and $\lambda =
(\lambda_1,\dots,\lambda_{d})$ with $\sum \lambda_i=dk$.  The
multiplicity of $\lambda$ in $S^\mu(S^k)$ equals
\begin{enumerate}
\item the multiplicity of $(\lambda_1-\lambda_d,\dots,\lambda_{d-1}-\lambda_d,0)$ in $S^\mu(S^{k-\lambda_d})$ if $\lambda_d$ is even,
\item the multiplicity of $(\lambda_1-\lambda_d,\dots,\lambda_{d-1}-\lambda_d,0)$ in $S^{\mu^\vee}(S^{k-\lambda_d})$ if $\lambda_d$ is odd, where $\mu^\vee$ is the transpose of $\mu$.
\end{enumerate}
Additionally, the value $\lambda_1-\lambda_d$ is determined by the
equation $d(k-\lambda_d) = \sum_{i=1}^{d-1} \lambda_i -
(d-1)\lambda_d$.  Consequently, our implementation uses arguments
\[
(b_1,\dots,b_{d-2}, s) =
(\lambda_{d-1}-\lambda_d,\dots,\lambda_2-\lambda_d,
k-\lambda_d).\] 
\end{rem}

\begin{rem}[Stable multiplicites]
Fix an integer $d$ and let $\lambda$ be a Young tableau.  For every
sufficiently large $k$, we can construct another tableau $\lambda'(k)$
by adding a new first row to $\lambda$ such that $|\lambda'(k)| = dk$.
As a function of $k$, the multiplicity of the isotypic component
$\lambda'(k)$ in $S^d(S^k)$ becomes eventually constant as $k$ grows
\cite{MR1037395, carre1992plethysm, brion1993stable, MR1651092}.  This
fact follows easily in our setting.  Indeed, note that the desired
multiplicity is a function of counts of $(\alpha,\lambda')$-matrices.
Now when $k$ is very large each possible filling of the columns $2$
to~$a$ of an $(\alpha,\lambda')$-matrix (restricted by the conditions
coming from $\lambda$) can be uniquely completed.
\end{rem}

\begin{rem}
Another possible approach to lattice point counting problems is
through Brion-Vergne formula \cite[p.~802
Theorem~(ii)]{brion1997residue} or \cite{baldoni2006volume} for vector
partition functions.  It provides an expression for the number of
lattice points in polytopes depending on shifts of facets.  Our
approach here is much more elementary.
\end{rem}

\begin{rem}\label{r:chambers}
The lattice point enumerators $Q_\alpha$ have chamber decompositions
into polyhedral cones.  We believe that the same fact also holds for
the whole expression in~\eqref{eq:formula}.  This can not be deduced
from the formula directly since shifting the arguments of $Q_\alpha$
by $\pi$ creates small bounded chambers.  This fact also complicates
our computations since the software is incapable of unifying chambers,
even if the quasi-polynomials on neighboring chambers agree.
Once there is theorem that guarantees a chamber decomposition into
cones, the computation should be revisited, because then the cones can
be computed in advance and the quasi-polynomials can be determined
using~\eqref{eq:formula}.  A possible approach to this problem is
outlined in Remark~\ref{r:symplectic}.
\end{rem}

\begin{rem}\label{r:symplectic}
A very general theory of Meinrenken and Sjamaar on the representation
theory of moment bundles on symplectic manifolds may be applicable to
the plethysm.
More specifically, let $M$ be a symplectic manifold with a Hamiltonian
action of a connected compact Lie group~$G$.  Let $L$ be a
$G$-equivariant line bundle and denote $RR(M,L)$ the push-forward of
$L$ to a point.  One may view $RR(M,L)$ either as a complex of
representations $H^i(M,L)$ with trivial derivations, or as an element
of equivariant $K$-theory (the representation ring).  Now let
$N^m(\lambda)$ denote the multiplicity of the representation
corresponding to a partition $\lambda$ in $RR(M,L^{\otimes m})$.
Corollary~2.12 in \cite{meinrenken1999singular} says that for every
moment bundle $L$ on $M$ the function $N^m(\lambda)$ (as a function of
$m$ and $\lambda$) is a piecewise quasi-polynomials with closed cones
as chambers.  In particular, each ray is contained in a single
chamber.

To get a result for plethysm, one has to find a suitable manifold $M$,
line bundle $L$ and group~$G$.  Results of
Brion~\cite{brion1993stable} show how to get the ingredients.  A
graded module structure on $\sum_k S^\mu(S^{k\nu}V)$ can be obtained
as follows. Let $X$ be the affine cone over the unique closed orbit in
$\p(S^\nu(V))$, i.e.,~$X=\Spec\left(\sum_k S^{k\nu}V \right)$.  Let
$T:=\{(t_1,\dots,t_{|\mu|})\subset (\C^*)^{|\mu|}:\prod t_i=1\}$ be
the $|\mu|-1$ dimensional torus.  The semidirect product
$\Gamma := S_{|\mu|} \ltimes T$ acts on
$\C[X]^{\otimes |\mu|}\otimes [\mu]$, where $[\mu]$ is the
representation of $S_{|\mu|}$.  The invariants of $\Gamma$ are
isomorphic, as a graded module, to~$\sum_k S^\mu(S^{k\nu}V)$.  Now
assume $\nu = (1)$ and $\mu = (|\mu|)$.  We obtain a graded algebra
structure on $\C[X]^{\otimes |\mu|}$ which in this case is just a
tensor power of a polynomial ring with $\Gamma$ and $C^* \times GL(V)$
actions.  In particular, the corresponding variety is smooth.  Since
the actions of $\Gamma$ and $\C^*\times GL(V)$ commute, we may
identify the isotypic component corresponding to $\lambda$ in the
plethysm with the invariants
$(\C[X]^{\otimes |\mu|} \otimes [\mu])^\Gamma$.  Hence, the space of
global sections of the line bundle $\mathcal{O}(k)$ on
$\Proj \left( \C[X]^{\otimes |\mu|} \right)$ acquires an additional
action of the finite group~$S_{|\mu|}$.  Meinrenken-Sjamaar's result
does not directly apply since the factor $S_{|\mu|}$ makes the group
nonconnected.  As pointed out to us by Michele Vergne, the theory
could be extended to this case. We leave this for future work, but the
feasibility has been demonstrated by Manivel, who used the above
method to get structural results about the asymptotics of Kronecker
coefficients \cite[Section 2.4]{manivel2014asymptotics}.
\end{rem}

\section{Asymptotic behavior}\label{sec:asymptotic}
Our main formula \eqref{eq:formula} also provides insight into the
asymptotical properties of plethysm. The main aim here is to identify
the leading terms of the piecewise quasi-polynomials that we
obtain. As already conjectured by Howe \cite{Howe87}, it is natural to
expect that the leading terms come from the polytope of highest
dimension, i.e.,~from the coefficient in the tensor product.  This is
not obvious since the contribution of a polytope in the
quasi-polynomial is \emph{not} of degree equal to the dimension of the
polytope.  The reason is the signed summations in the formula which
decrease the degree.  Below we show how to control this type of
cancelation, which allows us to obtain the asymptotics. Our strategy
is as follows:
\begin{enumerate}
\item Introduce a new variable $s$.
\item Multiply each variable in the quasi-polynomial by $s$ and ask
for the leading term with respect to the degree of $s$ in order to
identify the leading term.
\item Show that the contribution from polytopes of smaller dimension
is strictly smaller than the contribution from the
Littlewood--Richardson rule.
\end{enumerate}

More precisely, we compute the multiplicity of $s\lambda$ inside
$S^\mu(S^{sk})$ for $s\in \N$ for regular~$\lambda$, that is, when
$\lambda_i \neq \lambda_j$ for all $i\neq j$.  In this case, all
polytopes appearing in the computation of~\eqref{eq:formula} are
dilations of $P_{\alpha,\lambda}$ and~$P_{k,|\mu|}^\lambda$.  The
Hilbert-Ehrhart quasi-polynomials of these polytopes are particularly
important for us. We can compute the leading term of
$\# P_{sk,|\mu|}^{s\lambda}$, which is
$\vol P_{k,|\mu|}^\lambda s^{\dim P_{k,|\mu|}^\lambda}$.  One expects
this term to be the leading term of the entire formula, as the
dimension of $P_{\alpha,\lambda}$ is largest when
$\alpha = (1,\dots,1)$, the Littlewood--Richardson contribution.
Indeed, assume that $\alpha$ has $a$ parts and $\lambda$ has $l$
parts.  As we are only interested in partitions $s\lambda$, we can
assume that we work with exactly $l$ variables.  We have
$\dim P_{\alpha,\lambda}=(a-1)(l-1)$.  In contrast if
$\lambda=(\lambda_1^{a_1},\dots,\lambda_q^{a_q})$, with
$l=\sum_{j=1}^q{{a_j}}$, then
\[
\dim P_{k,d}^\lambda=1+\dots+(l-1)+(l-1)(d-l)-\sum_{j=1}^q{{a_j}\choose 2}=(l-1)(d-l/2-1)-\sum_{j=1}^q{{a_j}\choose 2}.
\]
We omit the easy but tedious proof of this fact as we do not need it
below.  Note that for regular $\lambda$, the dimension equals
$(l-1)(d-l/2-1)$.  One is tempted to conjecture, as in
\cite[3.6(d)]{Howe87}, that the leading term of the multiplicity of
the isotypic component corresponding to $s\lambda$ comes from
$\#P_{k,d}^\lambda$, as above.  It is obvious that this term appears,
due to the Littlewood--Richardson rule.  The main difficulty in
bounding the contributions from the other terms is that the counting
function is a piecewise quasi-polynomial: the shifts of argument by
the permutation $\pi$ may change both the chamber and the coefficients
of the polynomial.
We now provide the estimates for the function $\sum_{\pi\in
S_{d-1}}\sgn(\pi) Q_\alpha(k,\lambda_\pi)$.
\begin{lema}\label{l:qpolyasymptotics}
Suppose $\alpha$ has $a < d$ parts and $\lambda$ has $l$ parts.  The
leading coefficient of
\[
\sum_{\pi\in S_{l}}\sgn(\pi) Q_\alpha(sk,(s\lambda)_\pi)
\]
has degree strictly smaller than $(l-1)(d-l/2-1)$ with respect to the
variable~$s$.
\end{lema}
\begin{proof}
Suppose that $\alpha$ has $w$ parts greater than $1$ and $h$ parts
equal to $1$. In particular, $2w+h\leq d$. Each
$(\alpha, (s\lambda)_\pi)$-matrix $M$ is uniquely determined by two
matrices $(M_1, M_2)$, where $M_1$ is the $(w\times l)$-submatrix of
$M$, corresponding to rows with coefficients not equal to one and
$M_2$ the complimentary $(h\times l)$-submatrix.  Let $\alpha'$ be the
partition of $d-h$ obtained from~$\alpha$ by forgetting the
singletons, and let $\alpha_0:=(1,\dots,1) \vdash h$.  Introducing
parameters $i_j$ for $1\leq j\leq l-1$ corresponding to column sums of
$M_1$, we obtain:
\[
Q_\alpha(sk,(s\lambda)_\pi) =
\sum_{i_1=0}^{s\lambda_1+l}\dots\sum_{i_{l-1}=0}^{s\lambda_{l-1}+2}Q_{\alpha'}(sk,(i_1\dots
i_{l-1}))Q_{\alpha_0}(sk,(s\lambda)_\pi-(i_1\dots i_{l-1})).
\]
Note that, if $i_j>s\lambda_j+(l+1-j)-\pi(j)$ then
$Q_{\alpha_0}(sk,(s\lambda)_\pi-(i_1\dots i_{l-1}))=0$, so we could
restrict the summation indices; however, we prefer not to.  We obtain:
\begin{multline*}
\sum_{\pi\in S_{l}}\sgn(\pi) Q_\alpha(sk,(s\lambda)_\pi)= \\
\sum_{i_1=0}^{s\lambda_1+l}\dots\sum_{i_{l-1}=0}^{s\lambda_{l-1}+2}Q_{\alpha'}
(sk,(i_1\dots i_{l-1}))
\left(\sum_{\pi\in S_{l}}\sgn(\pi)Q_{\alpha_0}(sk,(s\lambda)_\pi-(i_1\dots i_{l-1}))\right).
\end{multline*}
We will bound
$\left|\sum_{\pi\in S_{l}}\sgn(\pi)
  Q_\alpha(sk,(s\lambda)_\pi)\right|$
by a polynomial in $s$ of small degree. To do this, it is enough to
bound
\[
\sum_{i_1=0}^{s\lambda_1+l}\dots\sum_{i_{l-1}=0}^{s\lambda_{l-1}+2}Q_{\alpha'}
(sk,(i_1\dots i_{l-1})) \left|\sum_{\pi\in
  S_{l}}\sgn(\pi)Q_{\alpha_0}(sk,(s\lambda)_\pi-(i_1\dots
  i_{l-1}))\right|.
\]
For any sequence of numbers $(\rho_i)_{i=1}^c$ of length $c$, let
$\varsigma \in S_c$ be a permutation that sorts~$\rho$,
i.e.,~$\rho_{\varsigma(i)}\geq \rho_{\varsigma(j)}$, for $i\leq j$.
We denote the sorted sequence by
$\varsigma (\rho) = (\varsigma_i(\rho))_{i=1}^c$.  Let $i_l$ be
defined by $\sum_{j=1}^l i_j=sk(d-h)$.  A technical problem in the
following argument is that the sequence $(s\lambda -(i_1\dots i_{l}))$
may be not ordered.
\begin{align*}
Q_{\alpha_0}\left(sk, (s\lambda)_\pi -(i_1\dots i_{l})\right) & =
Q_{\alpha_0}\left(sk, (s\lambda -(i_1\dots i_{l}))_\pi \right) \\
& = Q_{\alpha_0}\left(sk, \left(s\lambda_j - i_j + l - (j-1) -
\pi (j)\right)_{j=1}^{l-1}  \right) \\
& =Q_{\alpha_0} \left(sk, \left( \varsigma_t \left(\left(s\lambda_j - i_j + l -
    (j-1) \right)_{j=1}^{l} \right) - \pi (\varsigma (t)) \right)_{t=1}^{l-1}
\right).
\end{align*}
The purpose of this computation was simply to sort the arguments of
$Q_{\alpha_0}$ but with the additional complication that the sorting
acts on sequences of length $l$, while only $l-1$ arguments are used
by~$Q_{\alpha_0}$.  Now, if the sequence
$\varsigma \left(\left(s\lambda_j - i_j + l - (j-1) \right)_{j=1}^{l}
\right)$
has two equal entries, then $\sum_\pi \sgn(\pi) Q_\alpha(\cdot)$
vanishes because we can match up terms for the permutations differing
by the transposition exchanging the two corresponding indices.  Note
that we use the symmetry of the counting problem for $Q_\alpha$ too.
In this case, the bound holds trivially.  Now assume that the sequence
is strictly decreasing.  We get that
\[
\left(s\lambda_{\varsigma(j)} - i_{\varsigma(j)} + l - (\varsigma(j)-1) - (l - (j-1))\right)_{j=1}^{l}
\]
is nonincreasing.  Hence,
\begin{multline*}
\left|\sum_{\pi\in S_{l}}\sgn(\pi)
  Q_{\alpha_0}(sk,(s\lambda)_\pi-(i_1,\dots, i_{l-1}))\right| = \\
\left |\sum_{\pi\in S_{l}}\sgn(\pi) Q_{\alpha_0}\left(sk,
  \left(\left(s\lambda_{\varsigma(j)} - i_{\varsigma(j)} + l -
      (\varsigma(j)-1) - (l -
      (j-1))\right)_{j=1}^{l}\right)_{\pi}\right) \right|.
\end{multline*}
Now, by the arguments in Sections~\ref{sec:Littlewood-Richardson}
and~\ref{sec:ChangeBasis} this expression equals the multiplicity of
the isotypic component corresponding to the partition
\begin{equation}\label{eq:varsigma}
\left(s\lambda_{\varsigma(j)} - i_{\varsigma(j)} + l -
  (\varsigma(j)-1) - (l - (j-1))\right)_{j=1}^{l} =
\left(s\lambda_{\varsigma(j)} - i_{\varsigma(j)}  + j - \varsigma(j) \right)_{j=1}^{l}
\end{equation}
inside~$(S^{sk})^{\otimes h}$.  This allows us to bound the degree
with which $s$ may appear separately.  The degree of $s$ in the term
$Q_{\alpha'}(sk,(i_1\dots i_{l-1}))$ can be naively bounded by
$(w-1)(l-1)$, as each entry of the $j$-th column of $M_1$ is bounded
by $s\lambda_j$ plus a constant, and the row and column sums of $M_1$
are fixed.  It thus remains to bound the degree of $s$ in
\begin{equation}\label{eq:suma}
\sum_{i_1=0}^{s\lambda_1+l}\dots\sum_{i_{l-1}=0}^{s\lambda_{l-1}+2} \left|\sum_{\pi\in S_{l}}\sgn(\pi)
  Q_{\alpha_0}\left(sk, \left( \left(s\lambda_{\varsigma(j)} -
      i_{\varsigma(j)}  + j - \varsigma(j) \right)_{j=1}^{l} \right)_{\pi}\right)\right|.
\end{equation}
In the summation over the indices $i_j$, there may arise duplicates
among the partitions~\ref{eq:varsigma}.  However, the number of
duplicates of a given weight is bounded by $d!$, since for a fixed
$\varsigma$, there can be at most one sequence $(i_j)_j$ yielding a
given partition.  For any representation $W$, let $\mathfrak{s}_l(W)$
be sum of multiplicities of all isotypic components indexed by
partitions with at most $l$ parts.  We obtain that \eqref{eq:suma} is
bounded by $d!\mathfrak{s}_l((S^{sk})^{\otimes h})$.  It remains to
bound the $s$-degree of $\mathfrak{s}_l((S^{sk})^{\otimes h})$.  We
distinguish two cases depending on whether $h$ or~$l$ is larger.
\begin{description}[leftmargin=0cm, itemsep=1ex]
\item[Case~1 ($h\geq l$)] By Pieri's rule, the multiplicities of
isotypic components corresponding to partitions with at most $l$ parts
in $(S^{sk})^{\otimes h}$ are determined by the following parameters:
\begin{itemize}
\item One parameter for the number of boxes added to the first row in
the first step (the remaining boxes going into the second row),
\item Two parameters for the number of boxes added to the first and
second rows in step 2,
\item $i\leq l-1$ parameters for the number of boxes added to rows $1$
to $i$ in step~$i$,
\item $(h-l)(l-1)$ parameters for the numbers of boxes added to rows
from $1$ to $l-1$ in steps $l$ to~$h-1$.
\end{itemize}
This bounds the exponent of $s$ by
$1+\dots+l-1+(h-l-1)(l-1)=(l-1)(h-l/2-1)$. All together we obtain the
bound $(l-1)\left(w + h -1 -l/2-1\right)$ which is strictly smaller than
$(l-1)(d-l/2-1)$.

\item[Case~2 ($h< l$)] The degree of $s$ inside
\[
\sum_{i_1=1}^{s\lambda_1+l}\dots\sum_{i_{l-1}=1}^{s\lambda_{l-1}+2}\sum_{\pi\in
S_{l}} \sgn(\pi)Q_{\alpha_0}(sk,(s\lambda)_\pi-(i_1\dots i_{l-1}))
\]
is bounded by the degree of $s$ in the total sum of all multiplicities
in the decomposition of~$(S^{sk})^{\otimes h}$.  We could now proceed
as above, but~\cite[Theorem~1.2(ii)]{fulger2012asymptotic} directly
gives that this degree equals~$h\choose 2$, and hence, the total
degree in which $s$ can appear is at most
\[
(w-1)(l-1)+{h\choose 2}.
\]
After using $\binom{h}{2} \leq (l-1)(h-1)/2$ this is seen as strictly
smaller than $(l-1)(d-l/2-1)$. \qedhere
\end{description}
\end{proof}

\begin{thm}\label{thm:asympt}
Fix a partition $\mu$ of~$d$.  The multiplicity of the isotypic
component corresponding to $\lambda$ inside $S^\mu(S^k(V))$ is a
piecewise quasi-polynomial in $k$ and $\lambda$. In each
full-dimensional conical chamber, its highest degree term equals
$\frac {\dim \mu}{d!}$ times the highest degree term of the
multiplicity of $\lambda$ in $S^k(V)^{\otimes d}$.
\end{thm}
\begin{proof}
Let $\alpha \neq (1,\dots,1)$ and suppose that
$\sum_\pi \sgn(\pi) Q_\alpha (k, \lambda)$ has a leading term of
degree greater than or equal to $(l-1)(d-l/2-1)$.  Pick a $\lambda$
where the leading term does not vanish.  Using this $\lambda$ in
Lemma~\ref{l:qpolyasymptotics} yields a contradiction.  Now the result
follows since for regular $\lambda$ the contribution from
$\alpha=(1,\dots,1)$ is of degree $(l-1)(d-l/2-1)$, and each
full-dimensional conical chamber contains a regular~$\lambda$.
\end{proof}
\begin{con}
Let $\mu$ be a partition of $d$ and
$\lambda=(\lambda_1^{a_1},\dots,\lambda_q^{a_q})$ be a partition of
$kd$ with $l=\sum_{j=1}^q a_q$ parts. The multiplicity of $s\lambda$
in $S^\mu(S^{sk})$ is a quasi-polynomial in $s$ whose lead term has
degree
\[
\dim P_{k,d}^\lambda=1+\dots+(l-1)+(l-1)(d-l)-\sum_{j=1}^q{{a_j}\choose 2}=(l-1)(d-l/2-1)-\sum_{j=1}^q{{a_j}\choose 2},
\]
and coefficient $\frac{\dim \mu}{d!}\vol P_{k,d}$.
\end{con}
Note that the degree in the conjecture is an obvious upper bound and
that Theorem~\ref{thm:asympt} yields the the conjecture whenever
$\lambda$ is a regular partition, i.e.,~when $\lambda$ belongs to the
interior of the cone of valid parameters.

\begin{rem}
An interesting question asked by Mulmuley is whether counting
functions on individual rays are Ehrhart functions of some rational
polytopes.  In some cases, we can provide a negative answer using
reciprocity (cf.~\cite{King, briand2009reduced}).  For the details,
see~\cite{plethNotEhrhart}.
\end{rem}

\ifx\form\WITHAPP
\section{Appendix}
\label{sec:appendix}

\subsection{Vector partition functions}
\label{sec:vect-part-funct}
Consider a polyhedral cone in standard representation $\mathcal{C} =
\{Ax \geq b\} \subset \Q^{N}$, and a linear map $\pi: \mathcal{C} \to
\Q^{n}$.  The image of $\pi$ is a polyhedral cone denoted
$\mathcal{D}$.  In this situation, the preimage of an integral point
in $\mathcal{D}$ is a polyhedron and we are interested in the
(possibly infinite) number of integral points it contains.  The
\emph{counting function} is
\begin{align*}
  \phi  & : \mathcal{D}\cap\N^{n} \to \N_{0} \cup \{\infty\}  \\
  \phi (d) & = \# \{c \in \mathcal{C}\cap\N^{N} : \pi (c) = d\}
\end{align*}
The preimage of any rational point $d \in \mathcal{D}$ under $\pi$ is
a polyhedron and there are only finitely many combinatorial types of
polyhedra appearing among all preimages (see~\cite{Verdoolaege:2004aa}
for an overview on the history of this result with a focus on
implementation).  The type depends on which supporting hyperplanes of
$\mathcal{C}$ intersect a given preimage $\pi^{(-1)}(d)$
nontrivially.  This yields a decomposition of $\mathcal{D}$ known as
the \emph{chamber decomposition}.  As in each chamber the
combinatorial type of each fiber is the same, the counting function is
a quasi-polynomial since in general the fiber is a rational polytope.
All-together, $\phi$ is a piecewise quasi-polynomial.

In full generality Sturmfels has shown that the lattice point
enumerator of a parametric polyhedron $\{x: Ax \leq b(t)\}$ is a
piecewise quasi-polynomial in the parameters $t$, whenever $b(t) \in
\Z[t]$ is a linear polynomial~\cite{sturmfels1995vector}.  He calls
$\phi$ the \emph{vector partition function} as it counts the number of
ways to write a vector in terms of generators of~$\mathcal{C}$.

\begin{exm}
Fix $d\in\mathcal{D}$ and consider points $kd$, $k\in\N$ on the ray
generated by~$d$.  In this case $\phi(kd)$ equals $P_{\pi^{-1} (d)}
(k)$, the Ehrhart quasi-polynomial of $\pi^{-1}(d)$.  If $\pi^{-1}(d)$
happens to be an integral polytope, then so are the polytopes
$\pi^{-1}(kd)$ and in this case $P_{\pi^{-1}(d)}$ is an honest
polynomial~\cite{ehrhart1977polynomes}.
\end{exm}
\begin{exm}\label{ex:2dcounting}
  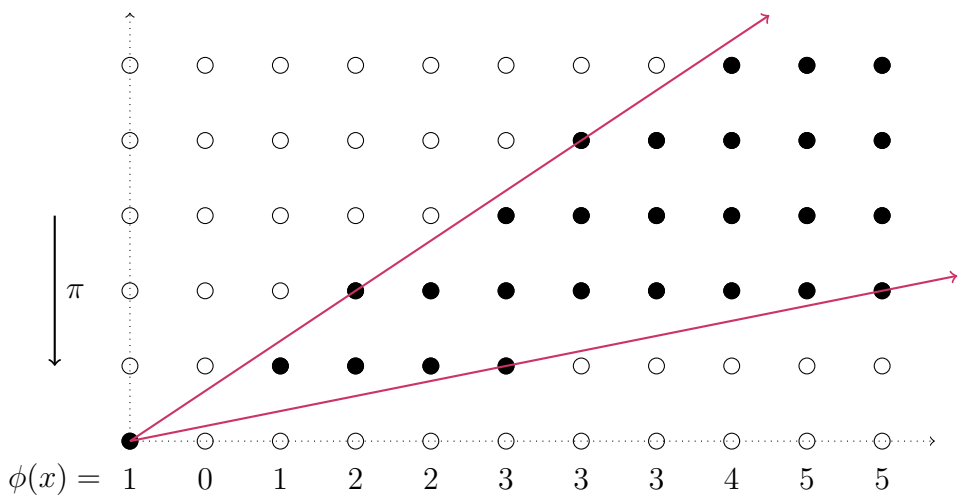
\begin{figure}[htpb]
  \centering
  \begin{tikzpicture}
  \def\cone{purple!80}
  \def\ctwo{green!40}
  \def\cthree{blue!40}

  \draw [->, dotted, thin] (0,0) -- (10.7,0);
  \draw [->, dotted, thin] (0,0) -- (0, 5.7);
  \foreach \i in {0,...,10} {
    \draw [black] (\i,0) circle (3pt);
  }
  \draw [->, black, thick] (-1,3) -- (-1,1) node [pos=0.5, right] {$\pi$};
  \node at (-1,-.5) {$\phi(x) = $};
  \node at (0,-.5) {$1$};
  \node at (1,-.5) {$0$};
  \node at (2,-.5) {$1$};
  \node at (3,-.5) {$2$};
  \node at (4,-.5) {$2$};
  \node at (5,-.5) {$3$};
  \node at (6,-.5) {$3$};
  \node at (7,-.5) {$3$};
  \node at (8,-.5) {$4$};
  \node at (9,-.5) {$5$};
  \node at (10,-.5) {$5$};

  \foreach \i in {0,...,10} {
    \draw [black]
    (\i,1) circle (3pt)
    (\i,2) circle (3pt)
    (\i,3) circle (3pt)
    (\i,4) circle (3pt)
    (\i,5) circle (3pt);
  }

  \filldraw [black] (0,0) circle (3pt);

  \foreach \i in {2,3,4,5} {
    \filldraw [black]
    (\i,1) circle (3pt);
  }
  \foreach \i in {3,...,10} {
    \filldraw [black]
    (\i,2) circle (3pt);
  }
  \foreach \i in {5,...,10} {
    \filldraw [black]
    (\i,3) circle (3pt);
  }
  \foreach \i in {6,...,10} {
    \filldraw [black]
    (\i,4) circle (3pt);
  }
  \foreach \i in {8,...,10} {
    \filldraw [black]
    (\i,5) circle (3pt);
  }

  \draw [->, \cone, thick] (0,0) -- (11,2.2) ;
  \draw [->, \cone, thick] (0,0) -- (8.5, 5.667);
\end{tikzpicture}
\captionsetup{singlelinecheck=off}
\caption[.]{The counting problem in Example~\ref{ex:2dcounting}.
  \label{fig:vectorpartition}}
\end{figure}
Consider the two-dimensional cone $\mathcal{C}$ over the matrix
$(\begin{smallmatrix} 5 & 3 \\ 1 & 2 \end{smallmatrix})$, depicted in
Figure~\ref{fig:vectorpartition}.  Let $\pi$ be the projection to the
first coordinate such that the image cone $\mathcal{D}$ is just the
$x$-axis.  The vector partition function counting the number of
integer points in a vertical slice of $\mathcal{C}$ is the piecewise
quasi-polynomial
\begin{equation*}
  \phi(x) =
  \begin{cases}
  0 &  x < 0, \\
  (x+1) - \lfloor \frac{x+2}{3} \rfloor - \lfloor \frac{x +
  4}{5}\rfloor & x \ge 0.
  \end{cases}
\end{equation*}
Note that in general a quasi-polynomial can be written as a polynomial
expression in the variables and floor functions of linear functions in
the variables.
\end{exm}

Vector partition functions can be computed symbolically.  The first
step is usually to compute the chamber decomposition for which several
algorithms exist.  Once the problem is reduced to determining a
quasi-polynomial for each chamber, interpolation may be the first idea
that comes to mind.  This is known as Clauss' method and was indeed
the first method suggested for the determinations of vector partition
functions.  This method has problems since it can be difficult to find
sufficiently many lattice points for interpolation.  A~more efficient
approach is to use Barvinok's method which, in the setting of vector
partition functions, was first suggested by Verdoolaege et
al.~\cite{verdoolaege2007counting}.  Their software \barvinok
(together with the \isl library) is the most advanced tool available
today.  The software is very well developed because of its
applications in computer science, for instance to loop optimization in
compiler development.  It is one of the mathematical software tools
with which run time is short compared to the time that humans need to
learn something from the result.  The introduction of
\cite{verdoolaege2005experiences} contains many references.

\subsection{Computation of the plethysm coefficient
quasi-polynomials}\label{sec:compute-qpolys}

As a proof of concept we evaluated equation~\eqref{eq:formula} using
\barvinok.  We describe the necessary steps for $d=5$ here.  The input
files for $d=3,4$ are also available on our project homepage.

In~\eqref{eq:formula} the innermost evaluation is the quasi-polynomial
function $Q_\alpha$ which depends on $\alpha$ (a partition of~$d$) and
a permutation $\pi$.  $Q_\alpha$ is a function in~$\lambda$, but
through the series of reductions in Section~\ref{sec:reductions}, the
final formula has different arguments, called $b_1,\dots,b_{d-1},s$
(see Remark~\ref{r:extend}).  By convention, the $b_i$ are ordered
increasingly with $b_1$ the smallest.

\begin{exm}
Suppose we want to determine the multiplicity of the isotypic
component for $\lambda=(3,2,1)$ in~$S^3(S^2)$.  This multiplicity is
equal to the multiplicity of the isotypic component of $(2,1,0)$ in
$\bigwedge^3(S^1)$ (Remark~\ref{r:extend}).  To evaluate it using our
programs, we plug $(b,s) = (1,1)$ into the quasi-polynomial stored in
\verb|111.qpoly|.  We find that the multiplicity is equal to~$0$.  See
the last item in Section~\ref{sec:quirks} for how this $0$ is
presented, though.
\end{exm}

In the following we describe the steps necessary to repeat our
computations and determine the quasi-polynomials
in~\eqref{eq:formula}.  The procedure consists of roughly four steps:
\begin{enumerate}
\item Determine $Q_\alpha$ enumerators with \verb|barvinok_enumerate|.
\item Sum over the partition $\alpha$ using precomputed coefficients
$\chi_\mu(\alpha)D_\alpha$.
\item Sum over permutations $\pi$.
\item Division by $\pm d!$ and postprocessing.
\end{enumerate}

\subsubsection{Enumeration.}  The directory
\verb|barvinok_enumeration| contains input files for the program
\verb|barvinok_enumerate|, corresponding to determination of the
$Q_\alpha$ for different values of $\alpha$ and shifted~$\lambda$.
These files need to be processed individually with the command
\begin{verbatim}
barvinok_enumerate --to-isl < "${input}" > "${input}".result
\end{verbatim}
where \verb|${input}| is a filename of one of the \verb|.barv| files.
We provide the bash script \verb|do.sh| which runs on four parallel
processors for this job.  In our experience, this step, for $d=5$,
should complete in less than 12 hours even on laptop computers.  At
this point one could argue that we are doing a lot of computation that
is not strictly necessary, since the $Q_\alpha (k,\lambda_\pi)$ for
different $\pi$ are the same quasi-polynomials, evaluated at slightly
shifted arguments.  In principle one would like to compute one
quasi-polynomial and evaluate it on shifted arguments.  This point is
legit, but it seemed more convenient using recomputation with modified
constraints.  Due to the parallelization, computing all $Q_\alpha
(k,\lambda_\pi)$ individually by modifying the input was quick.  The
result of this computation are stored in \verb|.result| files which we
need for the next, computationally more demanding, step: summation.

\subsubsection{Summation} At this point the \verb|.result| files
should contain all lattice point enumerators $Q_\alpha(k,\lambda_\pi)$
and our next task is to sum them with appropriate coefficients.  After
some experimentation it turned out to be advantageous to first sum
over the partition $\alpha$ and later over the permutation~$\pi$.  In
the light of Remark~\ref{r:chambers} the reason for this seems to be
the creation of many more small chambers during summation over~$\pi$.
For this step we use \iscc, an interactive (and scriptable) \isl
frontend which is distributed with \barvinok.  Place the result
files in the summation folder which already contains the appropriate
summation scripts.  Their names are \verb|sum11111.iscc| for
$\mu=(11111)$ and so on.  These scripts run for a while.  Here are
some approximate run times that we measured on an Intel Core i7-4770
(3.4GHz):
\begin{center}
\begin{tabular}{|c|c|}
\hline
script & runtime in hours \\\hline\hline
\verb|sum11111.iscc| & 118 \\
\verb|sum2111.iscc| & 38 \\
\verb|sum221.iscc| & 7 \\
\verb|sum311.iscc| & 2 \\
\verb|sum32.iscc| & 7 \\
\verb|sum41.iscc| & 38 \\
\verb|sum5.iscc| & 118 \\
\hline
\end{tabular}
\end{center}
In principle this computation could be parallelized too by structuring
summation hierarchically in the form of a tree.  At the moment \iscc
has no native support for parallelism so the only way to parallelize
this computation would be on the OS level.  This in turn means that
intermediate results have to be written to disk and read again.
Reading large quasi-polynomials is very slow (see
Section~\ref{sec:quirks}), and consequently we ran each summation on
one thread, but different summations at the same time.

The script \verb|sumX.iscc| stores its result in \verb|X.result|.
This file then contains the quasi-polynomial we are looking for,
multiplied with a signed factorial.  In the case of $d=5$, the factor
is $5! = 120$.

\subsubsection{Postprocessing}
\label{sec:postprocessing}
In the final step we divide the quasi-polynomial by the appropriate
factorial and sign and use a text editor to convert the results from
parametric sets of constant functions into functions (see
Section~\ref{sec:quirks}).

\subsection{Experiences and limitations} The results of our
computation are piecewise quasi-polynomials and their representations
are far from unique.  The most basic phenomenon bothering us is that
divisibility conditions may be obfuscated by existential quantifiers.
For one example, if $s$ is a variable, then $s \equiv 0 \mod 5$ may
appear as $\exists e_0 = \lfloor(-1 + s)/5 \rfloor, \exists e_1 =
\lfloor s/5 \rfloor$ such that $5e_1 = s$ and $5e_0 \leq -2 + s$ and
$5e_0 \geq -5 + s$.  To see the equivalence, note that the $e_0$
condition is that $s$ leaves any remainder except 1 modulo 5, and thus
redundant.  At the moment there seems to be no automatic way to remove
such redundant conditions while they do appear frequently (this one is
taken from Example~\ref{ex:case-study}).

Another challenge to be addressed in the future is the number of
chambers that appears after doing arithmetic with quasi-polynomials.
In principle there can be chambers $C_{1},C_{2}$ with corresponding
quasi-polynomials $p_{1},p_{2}$ such that $C_{2}$ is a face of
$C_{1}$, and $p_{1}$ when restricted to $C_{2}$ equals~$p_{2}$.  At
the moment \barvinok has no means to detect this case during the
computation, or rectify it a posteriori.  We can not precisely
estimate how much this effect hits us.  We did run \iscc's
\verb|coalesce| function on each of our results which uses simple
tests to detect empty chambers.  This has reduced the output of the
summation part to approximately a fifth of its original size.

\subsection{Quirks}
\label{sec:quirks}
Using \barvinok and \iscc, the following things occurred to us:
\begin{itemize}
\item It can take very long to read quasi-polynomials from disk.  Our
largest result files are \verb|11111.qpoly| and \verb|5.qpoly| which
on a Core~i7-4770 (3.4GHz) needed 5 hours and 51 minutes to be parsed.
In contrast they need only a second to be written to disk!  We asked
on the \isl development mailing list and it was confirmed that the
parser is not very efficient.
\item Reading a quasi-polynomial and then writing it out again need
not yield the same representation.  The parser that is used to read
piecewise quasi-polynomials from files applies certain transformations
that are not applied when computing the quasi-polynomials from
scratch.
\item Mathematically speaking our results are simply functions $\N^d
\to \N$, but in the computer things are not that simple.  The program
\verb|barvinok_enumerate| which we use as the first step in our
computation does not return functions on $\N^d$---it returns sets of
constant functions, parametrized over~$\N^d$.  It is technically
impossible to ``evaluate'' these parametric sets of constants, because
only functions can be evaluated in \isl.  To fix this we simply used
text editing on the output files to convert expressions like
\begin{verbatim}
[b1, s] -> { [] -> (1/2 * b1 + 1/2 * b1^2) : ... }
\end{verbatim}
into
\begin{verbatim}
{ [b1, s] -> (1/2 * b1 + 1/2 * b1^2) : ... }
\end{verbatim}
\item If the result of a quasi-polynomial evaluation is a nonzero
integer $n$, then the result is formatted as \verb|{n}|.  If, however,
the result is zero, the empty set is returned: \verb|{ }|.
\end{itemize}

\subsection{Evaluation}
Evaluation of explicit plethysm coefficients can be done in
\verb|LiE|~\cite{van1992lie} and other packages like \verb|SAGE|:

\begin{exm}\label{ex:sage}
To evaluate plethysm in \verb|SAGE|~\cite{sage}, first one sets up the
ring of symmetric functions in the Schur basis with
\begin{verbatim}
sage: s = SymmetricFunctions(QQ).schur()
\end{verbatim}
After this the (Schur function) plethysm can be computed by plugging
in as follows:
\begin{verbatim}
sage : s([2,1,1])(s[3,1])
\end{verbatim}
\end{exm}

For both parametric partial and complete evaluation of our stored
results, the most practical tool is \iscc.
\begin{exm}\label{ex:iscc-eval}
In \iscc, to evaluate a quasi-polynomial $P$ (created, for instance
with
\begin{verbatim}
P := read "111.qpoly";
\end{verbatim}
at arguments $(3,2)$ use the following input to \iscc:
\begin{verbatim}
P ({[3,2]});
\end{verbatim}
The $()$ brackets are used to trigger evaluation on an isl domain
introduced with $\{\}$ which in turn consists of only one isolated
point~$[3,2]$.
\end{exm}

\begin{exm}\label{ex:case-study}
In this example we explain how to arrive at the at the result in
Example~\ref{ex:largeeval} using the provided result files.  Note that
by means of the reductions in Remark~\ref{r:extend}, this formula
could in principle also be derived from the Cayley-Sylvester formula
in $SL_2(\C)$ representation theory, but this formula is not as
explicit as ours.  It involves counting tableaux under side
constraints.

Let again $\mu=(5)$, and $\lambda=(31,3,2,2,2)$.  The following code
loads the quasi-polynomial for $\mu$ from the file \verb|5.qpoly| and
evaluates it along the line $s\lambda$ for $s\in\Z$.  Note that the
read command will take very long (up to several hours) since the
parser for quasi-polynomials is not very optimized.
\begin{verbatim}
P:= read "5.qpoly";
{[s] -> [0,0,s, 6*s]} . P;
\end{verbatim}
The result looks like this:
\footnotesize
\begin{verbatim}
$2 := { [s] -> ((((((3/5 - 289/720 * s + 1/20 * s^2 + 1/720 * s^3) + (5/8 + 1/8 * s) *
floor((s)/2)) + (1/3 - 1/6 * s) * floor((s)/3)) + ((7/12 - 1/3 * s) + 1/2 *
floor((s)/3)) * floor((1 + s)/3) + 1/4 * floor((1 + s)/3)^2) + 1/4 * floor((s)/4))
- 1/4 * floor((3 + s)/4)) : exists (e0 = floor((-1 + s)/5): 5e0 = -1 + s and s >= 1);
[s] -> ((((((1 - 289/720 * s + 1/20 * s^2 + 1/720 * s^3) + (5/8 + 1/8 * s) *
floor((s)/2)) + (1/3 - 1/6 * s) * floor((s)/3)) + ((7/12 - 1/3 * s) + 1/2 *
floor((s)/3)) * floor((1 + s)/3) + 1/4 * floor((1 + s)/3)^2) + 1/4 * floor((s)/4)) -
1/4 * floor((3 + s)/4)) : exists (e0 = floor((-1 + s)/5), e1 = floor((s)/5): 5e1 = s
and s >= 5 and 5e0 <= -2 + s and 5e0 >= -5 + s); [s] -> (((((((-4/5 + 289/720 * s -
1/20 * s^2 - 1/720 * s^3) + (-5/8 - 1/8 * s) * floor((s)/2)) + (-1/3 + 1/6 * s) *
floor((s)/3)) + ((-7/12 + 1/3 * s) - 1/2 * floor((s)/3)) * floor((1 + s)/3) - 1/4 *
floor((1 + s)/3)^2) - 1/4 * floor((s)/4)) + 1/4 * floor((3 + s)/4)) * floor((s)/5) +
((((((4/5 - 289/720 * s + 1/20 * s^2 + 1/720 * s^3) + (5/8 + 1/8 * s) * floor((s)/2)) +
(1/3 - 1/6 * s) * floor((s)/3)) + ((7/12 - 1/3 * s) + 1/2 * floor((s)/3)) *
floor((1 + s)/3) + 1/4 * floor((1 + s)/3)^2) + 1/4 * floor((s)/4)) - 1/4 *
floor((3 + s)/4)) * floor((3 + s)/5)) : exists (e0 = floor((-1 + s)/5), e1 =
floor((s)/5): s >= 1 and 5e0 <= -2 + s and 5e0 >= -5 + s and 5e1 <= -1 + s and
5e1 >= -4 + s); [s] -> 1 : s = 0 }
\end{verbatim}
\normalsize To parse this, first observe that a new chamber starts
whenever we see \verb|[s] ->|.  The first step towards understanding
this output is to isolate the four chambers and to reformulate their
constraints.  The constraints are the items after the colon in each
chamber.

\begin{center}
\underline{Chamber 1}
\end{center}
\begin{verbatim}
exists (e_0 = floor((-1 + s)/5): 5e0 = -1 + s and s >= 1)
\end{verbatim}
which means $s\geq 1$ and $s \equiv 1 \mod 5$.

\begin{center}
\underline{Chamber 2}
\end{center}
\begin{verbatim}
exists (e0 = floor((-1 + s)/5), e1 = floor((s)/5):
5e1 = s and s >= 5 and 5e0 <= -2 + s and 5e0 >= -5 + s)
\end{verbatim}
which, except from $s\geq 5$, translates into the requirement that $s$
should leave remainder zero modulo $5$, and additionally $s-5 \leq
5\lfloor\frac{s-1}{5}\rfloor \leq s-2$.  The second condition is that
$s$ leaves any remainder except 1 modulo~5, and thus redundant.  At
the moment our computational tools are unable to carry out this
simplification automatically.

\begin{center}
\underline{Chamber 3}
\end{center}
\begin{verbatim}
exists (e0 = floor((-1 + s)/5), e1 = floor((s)/5): s >= 1 and 5e0
<= -2 + s and 5e0 >= -5 + s and 5e1 <= -1 + s and 5e1 >= -4 + s).
\end{verbatim}
The conditions are $s\geq 1$, $s-5 \leq 5\lfloor\frac{s-1}{5}\rfloor
\leq s-2$, and $s-4 \leq 5\lfloor\frac{s}{5}\rfloor \leq s-1$.  They are
both satisfied if and only if $s$ leaves remainder 2,3, or~4 modulo 5.

\begin{center}
\underline{Chamber 4}
\end{center}
This chamber is singleton: $s=0$ and thus the case distinction is
complete.

The output of our program has each quasi-polynomial written in an
expression involving floor functions.  To simplify the presentation,
let us introduce the following shorthands which appear in the output
\begin{gather*}
p_1 = \frac{3}{5} -  \frac{289}{720} s + \frac{1}{20} s^2 + \frac{1}{720} s^3,\\
p_2 = \frac{5}{8} +  \frac{1}{8} s, \qquad
p_3 = \frac{1}{3} -  \frac{1}{6} s, \qquad
p_4 = \frac{7}{12} - \frac{1}{3} s,  \\
q_1 = 1 - \frac{289}{720} s + \frac{1}{20} s^2 + \frac{1}{720}s^3, \\
r_1 = \frac{4}{5} - \frac{289}{720} s + \frac{1}{20} s^2 + \frac{1}{720} s^3.
\end{gather*}
Using these shorthands and only trivial manipulations of the output we
arrive at the following three quasi-polynomials in the three
nontrivial chambers:
\begin{gather*}
\underline{s \equiv 1 \mod 5}\\
p_1 + p_2 \left\lfloor \frac{s}{2}\right\rfloor + p_3
\left\lfloor\frac{s}{3}\right\rfloor
+ \left( p_4 + \frac{1}{2} \left\lfloor\frac{s}{3}\right\rfloor \right)
\left\lfloor\frac{1+s}{3}\right\rfloor
+ \frac{1}{4} \left( \left\lfloor\frac{1 + s}{3}\right\rfloor^2
+ \left\lfloor\frac{s}{4}\right\rfloor
- \left\lfloor\frac{3 + s}{4}\right\rfloor \right)\\
\underline{s \equiv 0 \mod 5}\\
q_1 + p_2  \left\lfloor\frac{s}{2}\right\rfloor
+ p_3 \left\lfloor \frac{s}{3} \right\rfloor +
\left( p_4 + \frac{1}{2}\left\lfloor \frac{s}{3} \right\rfloor \right)
\left\lfloor\frac{1 + s}{3}\right\rfloor
+ \frac{1}{4} \left( \left\lfloor \frac{1 + s}{3}\right\rfloor^2
+ \left\lfloor \frac{s}{4} \right\rfloor
- \left\lfloor \frac{3 + s}{4}\right\rfloor \right) \\
\underline{s \equiv 2,3,4 \mod 5}\\
r_1 + p_2 \left\lfloor \frac{s}{2} \right\rfloor + p_3 \left\lfloor\frac{s}{3}\right\rfloor +
\left( p_4 + \frac{1}{2} \left\lfloor\frac{s}{3}\right\rfloor \right)
\left\lfloor \frac{1 + s}{3} \right\rfloor
+ \frac{1}{4} \left(
  \left\lfloor \frac{1 + s}{3}\right\rfloor^2
  + \left\lfloor \frac{s}{4}\right\rfloor
  - \left\lfloor \frac{3 + s}{4}\right\rfloor \right)
%
\end{gather*}

There is an obvious pattern here, but unfortunately the \isl engine
has problems with factoring out, or simplifying these expressions
automatically.  For instance in the third chamber it actually returns
the expression
\small
\begin{gather*}
\underline{s \equiv 2,3,4 \mod 5}\\
\left (
r_1 - p_2 \left\lfloor \frac{s}{2} \right\rfloor - p_3 \left\lfloor\frac{s}{3}\right\rfloor -
\left( p_4 + \frac{1}{2} \left\lfloor\frac{s}{3}\right\rfloor \right)
\left\lfloor \frac{1 + s}{3} \right\rfloor
- \frac{1}{4} \left\lfloor \frac{1 + s}{3}\right\rfloor^2
- \frac{1}{4} \left\lfloor \frac{s}{4}\right\rfloor
+ \frac{1}{4} \left\lfloor \frac{3 + s}{4}\right\rfloor \right)
\left\lfloor \frac{s}{5} \right\rfloor + \\
\left(
-r_1 + p_2 \left\lfloor \frac{s}{2} \right\rfloor +
p_3 \left\lfloor\frac{s}{3}\right\rfloor
+ \left( p_4 + \frac{1}{2} \left\lfloor\frac{s}{3}\right\rfloor
\right) \left\lfloor \frac{1 + s}{3} \right\rfloor
+ \frac{1}{4} \left\lfloor \frac{1 + s}{3}\right\rfloor^2 +
\frac{1}{4} \left\lfloor \frac{s}{4} \right\rfloor
- \frac{1}{4} \left\lfloor \frac{3 + s}{4} \right\rfloor
\right) \left\lfloor \frac{3 + s}{5}\right\rfloor.
\end{gather*}
\normalsize Not only can we simplify the presentation, in fact the
above expression looks like the lead term would be of quasi-polynomial
nature while in reality it is not since
\begin{equation*}
\left (\left\lfloor \frac{s}{5} \right\rfloor - \left\lfloor \frac{3 +
    s}{5}\right\rfloor \right) = -1  \quad \text{ if } s \equiv 2,3,4
\mod 5.
\end{equation*}
Applying all simplifications and the shortcuts
\begin{gather*}
p = - \frac{289}{720} s + \frac{1}{20} s^2 + \frac{1}{720} s^3,\\
p_2 = \frac{5}{8} + \frac{1}{8} s, \qquad p_3 = \frac{1}{3} -
\frac{1}{6} s, \qquad p_4 = \frac{7}{12} - \frac{1}{3} s,  \\
A(s) = p + p_2 \left\lfloor \frac{s}{2}\right\rfloor + p_3
\left\lfloor\frac{s}{3}\right\rfloor + \left( p_4 + \frac{1}{2}
  \left\lfloor\frac{s}{3}\right\rfloor \right)
\left\lfloor\frac{1+s}{3}\right\rfloor + \frac{1}{4} \left(
  \left\lfloor\frac{1 + s}{3}\right\rfloor^2 +
  \left\lfloor\frac{s}{4}\right\rfloor - \left\lfloor\frac{3 +
    s}{4}\right\rfloor \right),
\end{gather*}
the final result is
\begin{equation*}
Q(s) = A(s) +
\begin{cases}
1 & \text{ if } s \equiv 0 \mod 5\\
\frac{3}{5} & \text{ if } s \equiv 1 \mod 5\\
\frac{4}{5} & \text{ if } s \equiv 2,3,4 \mod 5.
\end{cases}
\end{equation*}
\end{exm}
\fi

\bibliographystyle{amsalpha}
\bibliography{plethysm}

\end{document}